\documentclass[amsmath,amssymb,twoside,12pt]{article}
\usepackage{amsfonts, amssymb}

\usepackage{layout,verbatim,amssymb,amsmath,setspace,cite}
\newcommand{\erre}{\mbox{$\mathbb{R}$}}
\newcommand{\enne}{\mbox{$\mathbb{N}$}}

\providecommand{\keywords}[1]{\textbf{{Key Words:}} #1}

\usepackage{amsmath}
\usepackage{amssymb}
\usepackage{amsthm}
\newtheorem{theorem}{Theorem}[section]
\newtheorem{lemma}[theorem]{Lemma}
\newtheorem{proposition}[theorem]{Proposition}
\newtheorem{corollary}[theorem]{Corollary}
\newtheorem{definition}[theorem]{Definition}
\newtheorem{remark}[theorem]{Remark}

\begin{document}
\begin{center}
 {\bf \Large Comparison between some norm and order gauge integrals in Banach lattices}
\\
\vspace{0.2cm} Domenico Candeloro \\ University of Perugia (Italy) \\
Department of Mathematics and Computer Sciences \\
via Vanvitelli, 1
I-06123 Perugia, Italy\\
{\tt candelor@dipmat.unipg.it}\\
\vspace{0.2cm} Anna Rita Sambucini \\ University of Perugia (Italy) \\
Department of Mathematics and Computer Sciences \\
via Vanvitelli, 1
I-06123 Perugia, Italy\\
{\tt anna.sambucini@unipg.it}\\

\end {center}

\abstract{Henstock-type integrals are considered, for 
functions defined in a Radon measure space and taking values in a Banach lattice $X$ considering the norm and the order structure of the space. 
 A number of results are obtained, highlighting the differences between them, specially in the case of $L$-space-valued functions.\\}

\noindent {\bf AMS Subject Classification}:
28B05, 28B20, 46B42, 46G10, 18B15\\

\noindent \keywords{ Banach lattices,  $L$ and $M$ spaces, gauges, Henstock integral,  Mc Shane integral, variational Henstock integral}\\

\section{Introduction}
In the last century many different notions of integral were introduced 
for real-valued functions, in order to generalize the Riemann one. 
The wide literature in this topic attests the interest for this field of research: see for example 
\cite{f1994a,f1994b,f1995,riecan,SS,
b,r2005,r2009a,bcs2011,cao,DPMILL,DPVM,dpp,dr,dallas,gould,gcmed,
bdlibro,brv,KS,bs2003,bs2004,bss07,cc}. 
Afterwards the notions of order-type integral were also 
introduced and studied, for functions taking their values in 
suitable ordered vector spaces: see  \cite{bms,csmed,mimmoroma,dallas, panoramica,aBS2004,aBS2011}.\\
This paper aims  to compare some norm- and order-type integrals,  with some interesting results  when $X$ is an $L$-space.\\
Throughout this paper $(T,d)$ is a compact metric Hausdorff topological space, $\Sigma$ its Borel
$\sigma$-algebra and $\mu : \Sigma \to \mathbb{R}_0^+$ a regular, pointwise non atomic measure, and
 $X$ a  Banach (lattice) space.
\\
In Sections \ref{HMS-univoco} and \ref{n-o-single}
norm and order-type integrals are compared, showing 
a first showy difference: order integrals in general do not respect almost everywhere equality, except for order-bounded functions.
 Another noteworthy difference is that order integrals enjoy the  Henstock Lemma: this fact has interesting consequences in $L$-spaces, 
as shown in Theorem \ref{Lspazio}.\\
In the subsection  \ref{01single} integrability in $[0,1]$ is discussed and it is proven that monotone mappings are McShane order-integrable.
In the same framework, two further notions of order-type integrals, inspired at the paper \cite{nara}, have been introduced: they are the corresponding notions of the Henstock and McShane order-type integrals, but requiring measurability of the involved gauges: in the
 normed case, these notions are discussed in \cite{nara} where their strict connection with the Birkhoff integrability has been shown and their equivalence is proven, for bounded functions;  here we prove that, also for the order-type integral, for bounded functions the two notions are equivalent.

Some of these results have been exposed at the SISY 2014 - IEEE 12th International Symposium on Intelligent Systems and Informatics (see \cite{cs2014}) and they  will be applied to the multivalued case in a subsequent paper \cite{bcs2015}.

\section{Henstock (and McShane)  integrals}\label{HMS-univoco}
Given a compact metric Hausdorff topological space $(T,d)$ and its Borel
$\sigma$-algebra $\Sigma$, let $\mu: \Sigma \rightarrow \mathbb{R}_0^+$ be a  $\sigma$-additive (bounded) regular measure, so that $(T,d,\Sigma,\mu)$ is a Radon measure space. Let $X$ be a Banach space.
The following concept of Henstock integrability was presented in \cite{f1994a} for bounded measures in the Banach space context. 
See also  \cite{bms} for the following definitions and investigations.
\\
A \textit{gauge} is any map \mbox{$\gamma: T \rightarrow \mathbb{R}^+$.}
A \textit{decomposition}  $\Pi$  of $T$ is a finite family $\Pi = \{ (E_i,t_i): i=1, \ldots,k \} $ of pairs such that
$t_i \in E_i, \, E_i \in \Sigma$ and $\mu(E_i\cap E_j) = 0$ for $i \neq j$.
 The points $t_i$, $i=1, \ldots, k$, are called \textit{tags}.
If moreover $\bigcup_{i=1}^k \, E_i = T$, $\Pi$ is called a \textit{partition}.
A \textit{Perron partition} (or also \textit{Henstock partition}) is a partition for which each tag $t_i$ belongs to the corresponding set $E_i$; while, if this condition is not necessarily
 fulfilled, the partition is said to be {\em free}.\\
Given a gauge $\gamma$, $\Pi$ is said to be  \textit{$\gamma$-fine} $(\Pi \prec \gamma)$ if
$dist(w,t_i) < \gamma (t_i)$ for every $ w \in E_i$ and $i = 1, \ldots, k$.
\begin{definition}\rm  \label{fnorm}
\rm A function $f:T\rightarrow X$ is \textit{{\rm H}-integrable}
if there exists $I \in X$ such that, for every $\varepsilon > 0$
 there is a gauge \mbox{$\gamma : T \rightarrow \mathbb{R}^+$} such that for every $\gamma$-fine Perron partition of $T$,
$\Pi=\{(E_i, t_i), i=1, \ldots, q \}$, one has:
\begin{eqnarray*}
\left\|\sigma(f, \Pi)- I \right\| \leq \varepsilon,
\end{eqnarray*}
where the symbol $\sigma( f, \Pi)$ means
 $\sum_{i=1}^q  f(t_i) \mu(E_i)$.
In this case, the following notation will be used: $I = {\rm (H)}\int_T f d\mu$. 
\end{definition}

It is not difficult to deduce, in case $f$ is H-integrable in the set $T$, that also the restrictions $f\ 1_E$ are, for every measurable set $E$, thanks to the Cousin Lemma (see \cite[ Proposition 1.7]{riecan}).
This is not true in the classical theory, where  $T=[0,1]$ and the partitions allowed are only those consisting of sub-intervals.

\begin{definition}
\rm A function $f:T\rightarrow X$ is \textit{{\rm MS}-integrable}
if there exists $I \in X$ such that, for every $\varepsilon > 0$
 there is a gauge $\gamma : T \rightarrow \mathbb{R}^+$ such that for every $\gamma$-fine partition of $T$,
$\Pi=\{(E_i, t_i), i=1, \ldots, q \}$, one has:
\begin{eqnarray*}
\left\|\sigma(f, \Pi)- I \right\| \leq \varepsilon,
\end{eqnarray*}
where the symbol $\sigma( f, \Pi)$ means
 $\sum_{i=1}^q  f(t_i) \mu(E_i)$.
In this case, the following notation will be used: $I = {\rm (H)}\int_T f d\mu$. 
\end{definition}

 As it is well-known, in the
 classical theory H-integrability is different from McShane integrability and also from Pettis integrability.  Nevertheless, McShane integrability of a mapping $f:[0,1]\to X$ (here $X$ is any Banach space) is equivalent to Henstock and Pettis simultaneous integrabilities to hold (see \cite[Theorem 8]{f1994a}).\\

In the papers \cite{r2009a,dr,dpp} one can find an interesting discussion of the cases in which Henstock, Pettis and McShane integrability are equivalent, and also counterexamples in some related problems.
\\

Since the measure $\mu$ is assumed to be pointwise non atomic, a soon as partitions consist of arbitrary {\em measurable} sets,  all concepts in the Henstock sense will turn out to be equivalent to the same concepts in the McShane sense
(i.e. without requiring that the {\em tags} are contained in the corresponding sets of the involved partitions); however the {\em Henstock} terminology and notations will remain unchanged. 
The next proposition clarifies this fact.
\begin{proposition}\label{mettipunti}
Assume, in the previous setting, that $\mu$ is pointwise non atomic, i.e. $\mu(\{t\})=0$ for all $t\in T$.
If $f:T\to X$ is  {\rm H}-integrable in $T$, then it is also Mc Shane-integrable.
\end{proposition}
\begin{proof}
Of course, it will be sufficient to prove that, for every  gauge $\gamma$ and any $\gamma$-fine {\em free} partition  $\Pi^0$, there exists a Henstock-type $\gamma$-fine partition $\Pi^{\prime}$ satisfying $\sigma(f,\Pi^0) = \sigma(f,\Pi^{\prime})$.
\\
So, fix $\gamma$ and  $\Pi^0:=\{(B_i,t_i): i=1,...,k\}$ as above. Without loss of generality,  all the tags $t_i$ can be supposed to be distinct, otherwise it will be sufficient for each tag to take the union of all the sets $B_i$ paired with it.
Now, set $A:=\{t_i,i=1,...,k\}$ and define, for each $j$:
$B^{\prime}_j:=(B_j\setminus A)\cup\{t_j\}$.
Of course, each $B^{\prime}_j$ is measurable and is contained in $\gamma(t_j)$. Then the pairs $(B_j,t_j)$ form a $\gamma$-fine Henstock-type partition $\Pi^{\prime}$ and $\mu(B^{\prime}_j)=\mu(B_j)$ for all $j$, so
$\sigma(f,\Pi^0) = \sigma(f,\Pi^{\prime})$. This concludes the proof.
\end{proof}
So this Proposition shows that the use of free
$\gamma$-fine partitions does not modify the integral.\\

From now on  suppose that  $X$ is a Banach lattice
 with order-continuous norm, $X^+$ is its positive cone and $X^{++}$ is the subset of strictly positive elements of $X$.
The symbols $|~ |,\, \| ~\|$ refer to modulus and norm of $X$; for the relation between them and applications in Banach lattices see for example 
\cite{bvg,bvl,fvol3,bms,bbs,LW2,mn}.\\
An element $e$ of $X^+$ is an order unit  if for every $u \in X$ there is an $n \in \mathbb{N}$ such that $|u| \leq  ne$.\\ 
An $L$-space is a Banach lattice $L$ such that $\|u + v\| = \|u\| + \|v\|$ whenever $ u, v\in  L^+$.\\
An $M$-space is a Banach lattice $M$ in which the norm is an order-unit norm, namely there is an order unit $e$ and an equivalent Riesz norm $\| \cdot \|_e $ defined on $M$ by the formula
$\|u\|_e := \min\{ \alpha: |u| < \alpha e \}$
for every $u \in M$. In this case one has also
$\|u+v\|=\|u\|\vee \|v\|$. For further properties about
$L$- and $M$-spaces see also \cite[Section 354]{fvol3}.  \\
 
In this setting the notion of {\em order-type} integral can be given.
\begin{definition}\rm  \label{forder}
\rm A function $f:T\rightarrow X$ is \textit{order-integrable}
in the Henstock sense (\textit{{\rm (oH)}-integrability}  for short) if there exists $J \in X$ together with an $(o)$-sequence $(b_n)_n$ in $X$ and a corresponding sequence $(\gamma_n)_n$ of gauges, such that for every $n$ and every $\gamma_n$-fine Henstock partition of $T$,
$\Pi=\{(E_i, t_i), i=1, \ldots, q \}$, 
it is
$\left|\sigma(f, \Pi)- J \right| \leq b_n$
and the integral $J$ will be denoted with (oH)$\int f$.
\end{definition}

\begin{remark}\rm
Also in this case there is no difference in taking all free
$\gamma_n$-fine partitions.
It is easy to see that any (oH)-integrable  $f$ is also H-integrable and the two integrals coincide
thanks to the order continuity of the norm in $X$. The converse implication holds when $X$ is an $M$-space, but not in general: see for example Theorem \ref{Lspazio} and the following remarks.
Moreover  the (oH)-integrability of a function  $f$ implies also
its  Pettis integrability, thanks to  \cite[Theorem 8]{f1994a}.
\end{remark}

The following  Cauchy-type criterion holds:
\begin{theorem}\label{ordercauchy}{\rm (\cite[Theorem 5]{cs2014})}
Let $f:T\rightarrow X$ be any function. Then $f$ is {\rm (oH)}-integrable  if and only if there exist an $(o)$-sequence $(b_n)_n$ and a corresponding sequence $(\gamma_n)_n$ of gauges, such that for every $n$, as soon as $\Pi, \Pi'$ are two $\gamma_n$-fine Henstock  partitions, the following holds:
\begin{eqnarray}\label{cauchyprimo}
|\sigma(f,\Pi)-\sigma(f,\Pi')|\leq b_n.
\end{eqnarray}
\end{theorem}

\section{Properties of the  Norm and the  Order integrals}\label{n-o-single}
Some interesting properties of the order integral can be described  comparing it with the norm one.
A first singular fact is that, in general, almost equal functions  
may have a different behavior with respect to the (oH)-integral, as  proven in \cite[Example 2.8]{bms}.
However, for order-bounded functions,  the following result holds.

\begin{proposition}\label{bdd} {\rm \cite[Proposition 4]{cs2014}}
Let $f,g:T\to X$ be two bounded maps, such that $f=g$ $\mu$-almost everywhere.  Then, $f$ is {\rm (oH)}-integrable if and only if $g$ is, and the integrals coincide.
\end{proposition}

Another interesting difference concerns the Henstock Lemma, which is valid in the present form for the order integral but not, in general, for the norm one.
The technique of the proof  is
similar to that of  \cite[Theorem 1.4]{mimmoroma}.

\begin{proposition}\label{henlemma}{\rm \cite[Proposition 6]{cs2014}}
Let $f:T\to X$ be any {\rm (oH)}-integrable  function. Then, there exist an $(o)$-sequence $(b_n)_n$ and a corresponding sequence $(\gamma_n)_n$ of gauges, such that, for all $n$ and all $\gamma_n$-fine Henstock  partitions $\Pi$ one has
\[ \sum_{E\in \Pi} \sup_{\Pi'_E,\Pi''_E}\{|\hskip-2mm\sum_{F''\in \Pi''_E}f(\tau_{F''})\mu(F'') - \hskip-2mm\sum_{F'\in \Pi'_E}f(\tau_{F'})\mu(F')|\} \leq b_n
\]
and $\Pi'_E, \Pi''_E$ 
belong to the family of all $\gamma_n$-fine Henstock partitions of $E$.
\end{proposition}

\begin{proof}  
By theorem \ref{ordercauchy}, an $(o)$-sequence $(b_n)_n$ exists, together with a corresponding sequence of gauges $(\gamma_n)_n$, so that
\begin{eqnarray}\label{primocauchy}
\sum_{F'\in \Pi'}f(\tau_{F'})\mu(F')-\sum_{F''\in \Pi''}f(\tau_{F''})\mu(F'')\leq b_n
\end{eqnarray}
holds, for all $\gamma_n$-fine partitions $\Pi', \ \Pi''$. Now, 
consider any $\gamma_n$-fine partition $\Pi$, and, for every $E$ in  $\Pi$, 
take two arbitrary subpartitions $\Pi'_E$ and $\Pi''_E$.\\
By varying $E$, this gives rise to 
 two $\gamma_n$-fine partitions of $T$ 
for which (\ref{primocauchy}) holds true. 
Denote by $E_1$ the first element of $\Pi$, and, 
in the summation at left-hand side, 
take the supremum with respect just to the sets of the type $F'$  that are contained in $E_1$:
this yields
$$\sup_{\Pi'_{E_1}}\sigma(f,\Pi'_{E_1})+\sum_{\substack{F'\in \Pi',\\ F'\not\subset E_1}}f(\tau_{F'})\mu(F')-\sum_{F''\in \Pi''}\sigma(f,\Pi)\leq b_n.$$
Now, by varying only the $F''s$ that are contained in $E_1$, one obtains:
\begin{eqnarray*}
 \sup_{\Pi'_{E_1}}\sigma(f,\Pi'_{E_1})-\inf_{\Pi''_{E_1}}\sigma(f,\Pi''_{E_1})+
 \sum_{\substack{F'\in \Pi', \\ F'\not\subset E_1}}f(\tau_{F'})\mu(F')-\sum_{F''\not\subset E_1}\sigma(f,\Pi)\leq b_n,
\end{eqnarray*}
namely
\begin{eqnarray*}
 \omega_n(f,E_1)+ \hskip-2mm \sum_{\substack{F'\in \Pi',\\ F'\not\subset E_1}}f(\tau_{F'})\mu(F') - \hskip-2mm
 \sum_{\substack{F''\in \Pi'',\\F''\not\subset E_1}}f(\tau_{F''})\mu(F'')\leq b_n.
\end{eqnarray*}
where $$\omega_n(E):=\sup_{\Pi'_E,\Pi''_E}\{|\hskip+0mm\sum_{F''\in \Pi''_E}f(\tau_{F''})\mu(F'') - \hskip+1mm\sum_{F'\in \Pi'_E}f(\tau_{F'})\mu(F')|\}.$$

Now, starting from the last inequality, make the same operation in  the second subset of $\Pi$ (say $E_2$), by keeping fixed all the $F'$ and $F''$ that are not contained in it: so
\begin{eqnarray*}
\hskip-1cm&& \omega_n(f,E_1)+\omega_n(f,E_2)+\hskip-2mm \sum_{\substack{F'\in \Pi',\\ F'\not\subset E_1\cup E_2}}f(\tau_{F'})\mu(F')
- \sum_{\substack{F''\in \Pi'',\\F''\not\subset E_1\cup E_2}}f(\tau_{F''})\mu(F'')\leq b_n.
\end{eqnarray*}
Continuing this procedure the assertion follows.
\end{proof}

A consequence of Proposition \ref{henlemma} is that  
(oH)-integrability is hereditary on every  measurable subset $A$ of $T$.
In fact taking the same $(o)$-sequence $(b_n)_n$ 
together  with the corresponding sequence $(\gamma_n)_n$ 
for each $n$ any  $\gamma_n$-fine partition of $A$ can be extended to a $\gamma_n$-fine partition of $T$ 
thanks to the Cousin Lemma and so, for any two $\gamma_n$-fine partitions $\Pi, \ \Pi'$ of $A$, it follows
$$|\sigma(f,\Pi)-\sigma(f,\Pi')|\leq \omega_n(f,A)\leq b_n.$$
Then, the Theorem \ref{ordercauchy} yields the conclusion. The additivity of the integral can be obtained as well.
Moreover by Proposition \ref{henlemma} it follows:

\begin{corollary}\label{henstoc2}{\rm (\cite[Theorem 9]{cs2014})}
Let $f:T\to X$ be any {\rm (oH)}-integrable function. Then there exist an $(o)$-sequence $(b_n)_n$ and a corresponding sequence $(\gamma_n)_n$ of gauges, such that:
\begin{description}
\item[\ref{henstoc2}.1)]
 for every $n$ and every $\gamma_n$-fine partition $\Pi$  one has
$$\sum_{E\in \Pi}|f(\tau_E)\mu(E)-{\rm (oH)}\int_Ef d\mu|\leq b_n.$$
\item[\ref{henstoc2}.2)]
for every $n$ and every $\gamma_n$-fine partition $\Pi$ it holds
$$\sum_{E\in \Pi}|f(\tau_E)\mu(E)-f(\tau_E')\mu(E)|\leq b_n,$$
when all the tags  satisfy the condition $E\subset \gamma_n(\tau_E')$ and $E\subset \gamma_n(\tau_E)$ for all $E$.
\end{description}
\end{corollary}

\begin{remark}\label{paracool}\rm
Observe that in  Corollary \ref{henstoc2}  all partitions may also be
free, since, as  already noticed, the restriction $\tau_E\in E$ does not affect the results.
\end{remark}

A consequence of this theorem is  the following
\begin{theorem}\label{modulointegrabil}{\rm (\cite[Theorem 11]{cs2014})}
If $f:T\to X$ is {\rm (oH)}-integrable, then also $|f|$ is.
\end{theorem}

Now one can prove that the  modulus of
 the indefinite  oH-integral of $f$ is precisely the indefinite oH-integral of $|f|$; a result of this type was given in  
\cite[Theorem 1]{DW}
for the Bochner integral  of  functions taking values in a Banach lattice $X$.
\begin{definition}\label{modulus}\rm 
Let $f:T\to X$ be any {\rm (oH)}-integrable mapping, and set
$\mu_f(A)={\rm (oH)}\int_A f d\mu$
for all Borel sets $A\in \mathcal{B}$.
 Then $\mu_f$ is said to be the {\em indefinite integral} of $f$. The {\em modulus} of $\mu_f$, denoted by $|\mu_f|$, is defined for each $A\in \mathcal{B}$ as follows:
$|\mu_f|(A)=\sup\{\sum_{B\in \pi}|\mu_f(B)|: \pi\in \Pi(A)\}$
where $\Pi(A)$ is the family of all finite partitions of $A$.
\end{definition}
Then
\begin{theorem}\label{parallelo} {\rm (\cite[Theorem 13]{cs2014})}
If $f:T\to X$ is {\rm (oH)}-integrable then one has
$|\mu_f|= \mu_{|f|}.$
\end{theorem}

\begin{proof} 
The {\rm (oH)}-integrability of $|f|$ follows from 
 Theorem \ref{modulointegrabil},
together with
$|\mu_f|\leq \mu_{|f|}$ and so 
$|\mu_f|$ is bounded.
 Now  for the reverse inequality it will be sufficient to prove that $\mu_{|f|}(T)=|\mu_f|(T)$
thanks to the additivity of  $|\mu_f|$ and $\mu_{|f|}$.
 Let $(b_n)_n$ and $(\gamma_n)_n$  be an $(o)$-sequence and its corresponding sequence of gauges
 such that, for every $n$ and every $\gamma_n$-fine partition $\pi\equiv (E_i,t_i)_i$ it holds
\begin{eqnarray}
 &&\mbox{~}\hskip-1.5cm \nonumber \vee  \left\{
\sum_i\left|f(t_i)\mu(E_i)-{\rm (oH)}\int_{E_i}f d\mu\right|,
\sum_{i}\left| |f(t_i)|\mu(E_i)-{\rm (oH)}\int_{E_i}|f |d\mu\right|
\right\}
\leq b_n,\\
&& \label{terz}
\sum_{i}\left|\ |f(t_i)|\mu(E_i)-\left|{\rm (oH)}\int_{E_i}f d\mu \right|\ \right|\leq b_n.
\end{eqnarray}
 So, 
\begin{eqnarray*}
\mu_{|f|}(T)-|\mu_f|(T) &\leq& \sum_i\left(\mu_{|f|}(E_i)-\left|{\rm (oH)}\int_{E_i}fd\mu \right|\right) \leq 
\\&\leq&
 \sum_i\left (\mu_{|f|}(E_i)-|f(t_i)|\mu(E_i)\right)+\\
&+& 
\sum_i\left (|f(t_i)|\mu(E_i)-\left|{\rm (oH)}\int_{E_i}fd\mu\right|\right)\leq 2 b_n
\end{eqnarray*}
thanks to  (\ref{terz}). Since $ b_n  \downarrow 0$, then 
$\mu_{|f|}(T)=|\mu_f|(T)$, from which 
the assertion follows.
\end{proof}

Another  appealing consequence  is obtained in $L$-spaces.
In the sequel, when the Banach lattice is assumed to be an $L$-space, it will be denoted by $L$ rather than $X$.
The following definition, related with norm-integrability, is introduced.
\begin{definition} \rm \cite[Definition 3]{DPVM}
 $f:T \to X$ is {\em variationally {\rm H} integrable} (in short {\rm vH}-integrable) if 
for every $\varepsilon>0$ there exists a gauge $\gamma$ such that, for every $\gamma$-fine partition $\Pi\equiv(E,t_E)_E$ the following inequality holds:
\begin{eqnarray}\label{variazio}
\sum_{E\in \Pi}\|f(t_E)\mu(E)-{\rm (H)}\int_E f d\mu\|\leq \varepsilon.
\end{eqnarray}
\end{definition}
For results on  variational integrability see also \cite{M1,DPVM,bdpm}. 
\begin{theorem}\label{Lspazio}{\rm (\cite[Theorem 15]{cs2014})}
If $f:T\to L$ is {\rm (oH)}-integrable then  $f$ is Bochner integrable.
\end{theorem}

\begin{proof} 
It will be enough to prove that  $f$ is vH-integrable thanks to \cite[Theorem 2]{DPMILL}.
By  Corollary \ref{henstoc2}, there exist an $(o)$-sequence $(b_n)_n$ and a corresponding sequence $(\gamma_n)_n$ of gauges, such that, for 
every $n$ and every $\gamma_n$-fine partition $\Pi$  one has
$$\sum_{E\in \Pi}|f(\tau_E)\mu(E)-f(\tau_E')\mu(E)|\leq b_n.$$ 
By order continuity of the norm it is
$\lim_n\|b_n\|=0$. So, fix $\varepsilon>0$ and let $N \in \mathbb{N}$  be such that $\|b_N\|\leq \varepsilon$.
Then, if $\Pi$ is any $\gamma_N$-fine partition, one has
$$\left\| \, \sum_{E\in \Pi}|f(\tau_E)\mu(E)-f(\tau_E')\mu(E)| \, \right\| \leq \|b_N \| \leq \varepsilon,$$
in accordance with  condition {\bf \ref{henstoc2}.2}). 
Since $\|\cdot \|$ is an $L$-norm  it follows that
\begin{eqnarray}\label{normaepsilo}
\sum_{E\in \Pi} \, \big|\|f(\tau_E)\|-\|f(\tau_E')\|\big|\mu(E) \, \leq \|b_N\|\leq \varepsilon,\end{eqnarray}
when $\Pi$ is $\gamma_N$-fine, both for the tags $\tau_E$ and for the tags $\tau_E'$. 
Now, proceeding as in  Theorem \ref{modulointegrabil}, it is not difficult to prove that $\|f\|$ satisfies the Cauchy criterion  for the Henstock integrability, and therefore it is integrable.
From (\ref{normaepsilo}) and  (\ref{henstoc2}.1) the condition (\ref{variazio}) follows.
 \end{proof}

This result is apparently in contrast with the common situation in general Banach spaces (when only norm integrals are involved); moreover, in some normed
lattices,  Bochner (norm) integrability does not imply {\rm oH}-integrability, as shown in the example
 (\cite[Example 2.8]{bms}).
\\

Theorem \ref{Lspazio} can be used to show that  the  (oH)-integrability is stronger than  (H)-integrability: 
it is enough to consider for  example the function $f$ defined in  \cite[Theorem 3]{SS} when $X$ is any $L$-space of infinite dimension. The function $f$ is McShane integrable (i.e. H-integrable), but not Bochner integrable. Since $f$ takes values in an $L$-space, it cannot be oH-integrable, in view of the previous theorem.
\\

Another consequence of the previous results is that an $M$-space $X$ admits an equivalent $L$-norm only if it is of finite dimension. In fact in such an $M$-space $X$ 
the  (H)-integrability and the (oH)-integrability are the same: so, if $X$ has an equivalent $L$-norm then the previous argument
 immediately shows that $X$ must be of finite dimension.

\subsection{Integrability in $[0,1]$}\label{01single}
In this section 
the space $T$ is 
 $[0,1]$  endowed with the usual Lebesgue measure $\lambda$. 
Rather than using partitions made up with arbitrary measurable subsets, here they will be taken more traditionally as {\em free} partitions consisting of subintervals.
Indeed, in \cite{f1995} it is proven that there is equivalence between the two types:
 though the proof there is related only to norm integrals, the technique is the same. 
The symbol (oM)-integral will be used  instead of (oH)-integral. 
Analogously to  \cite[ Lemma 5.35]{KS}, it holds: 
\begin{lemma}\label{come5.35}{\rm (\cite[lemma 18]{cs2014})}
Let $f:[0,1]\to X$ be given  and suppose that there exists an $(o)$-sequence $(b_n)_n$ such that, for every $n$ 
two {\rm (oM)}-integrable functions $g_1$ and $g_2$ can be found, with the same regulating $(o)$-sequence 
$(\beta_n)_n$, such that $g_1\leq f\leq g_2$ and 
$${\rm (oM)}\int g_2 d\lambda \leq {\rm (oM)}\int g_1 d\lambda+b_n.$$
Then $f$ is {\rm (oM)}-integrable.
\end{lemma}

The (oM)-integrability  of  increasing functions   can be obtained  as in \cite[example 5.36]{KS}.
\begin{theorem}\label{mcmonotone}{\rm (\cite[Theorem 19]{cs2014})}
Let $f:[0,1]\to X$ be increasing. Then $f$ is {\rm (oM)}-integrable. 
\end{theorem}

\medskip
In a similar way, one can prove  the  (oM)-integrability more generally for (order bounded) mappings that are Riemann-integrable in the order sense. 
\\

Above and also in the previous Theorem, one can easily see that the gauges involved in the proof are constant mappings. 
More generally, one can consider notions of McShane, or Henstock integrability, in which the gauges involved are required to be {\em measurable} positive mappings. These notions have been studied in \cite{nara} and related bibliography, in terms of the {\em norm} integral.
The corresponding notions for the order-type integral are now considered. 


\begin{definition}\label{naram}\rm
 The mapping $f:[0,1]\to X$ is said to be {\em order-Birkhoff integrable} if there exist an element $J\in X$, an $o$-sequence $(p_n)_n$ in $X$ and a corresponding sequence of {\em measurable}  gauges $\gamma_n$ on $[0,1]$ such that, for every integer $n$, and every $\gamma_n$-fine free partition $P\equiv (t_i,E_i)_{i=1}^{k}$ of $[0,1]$ it holds
$|\sigma(f,P)-J|\leq p_n,$
where $\sigma(f,P)=\sum_{i=1}^kf(t_i)|E_i|.$
\end{definition}

Observe  that the {\em partitions} $P$ consist of non-overlapping intervals, and the tags $t_i$ do not belong necessarily to the corresponding subintervals $E_i$. 
When $f$ is order-Birkhoff integrable (shortly, (oB)-integrable), the element $J$ will be called the oB-integral of $f$ and denoted by (oB)$\int_0^1 f dx$.
The terminology used here, i.e. {\em Birkhoff integrability}, is due to the fact that, in the case of norm-type integrals, this notion is equivalent precisely to the Birkhoff (norm) integral: see \cite[Remark 1]{nara}.
\\
For further reference, the notation (oBH)-integrability (i.e. {\em (order-Henstock-Birkhoff)}-integrability)  is used if, in the above definition, only Perron-type partitions are allowed, i.e. partitions for which the tags belong to the corresponding sub-intervals.
Standard techniques allow to prove usual properties of the (oB)-integral, according to the next Proposition.

\begin{proposition}
Let $f$ and $g$ be $oB$-integrable mappings on $[0,1]$, and fix arbitrarily $\alpha, \beta$ in $\erre$. Then $\alpha f+\beta g$ is  {\rm oB}-integrable and
$${\rm (oB)}\int_0^1 (\alpha f+\beta g) dx =\alpha\ {\rm (oB)}\int_0^1  f\ dx+\beta {\rm (oB)}\int_0^1  g\ dx.$$
Moreover, $f1_{[a,b]}$ is {\rm oB}-integrable, for every $a,b \in [0,1],\ a<b$, and
$${\rm (oB)}\int_0^1 f\ dx =  {\rm (oB)}\int_0^1 f 1_{[0,a]}\ dx +{\rm (oB)}\int_0^1 f 1_{[a,1]}\ dx $$
Clearly, the notation ${\rm (oB)}\int_a^b f \ dx$ will be used, to indicate the integral of $f1_{[a,b]}$.
\end{proposition}
\vskip.2cm

In  \cite[Example 2.1]{r2009b}, it is proven that, 
in general, H-integrable mappings are not Birkhoff (neither, a fortiori, Bochner) integrable: since the mapping given there takes values in an $M$-space (i.e. $l^{\infty}([0,1])$), this is also an example of an oM-integrable function that is not Birkhoff integrable. This 
 example will be used to discuss other implications. First, denote by $D$ the set of all  dyadic rational points in $[0,1]$ and enumerate the elements of $D$ as a sequence $(d_n)_{n=0}^{+\infty}$. 
Next, set $\varphi(d_n)=\frac{1}{2^n}$ for each $n$. Then define
the following $\sigma$-additive measure $\mu$ in the space $([0,1],\mathcal{P}([0,1]))$:
$$\mu(E):=\sum_{d_n\in E}\varphi(d_n),$$
for each $E\subset [0,1]$. Clearly $\mu$ is finite and the space $L:=L^1([0,1], \mathcal{P}([0,1]),\mu)$ is an $L$-space, with the norm
$$\|f\|_1=\sum_{d\in D}|f(d)|\varphi(d).$$
Observe that each function in $l^{\infty}([0,1]$ is also a member of $L$, since the series $\sum_n\varphi(d_n)$ is convergent, and the two spaces $L$ and $l^{\infty}$ have the same natural ordering.
Moreover, if $(p_n)_n$ is a decreasing sequence in $l^{\infty}$ with infimum $0$, it is clear that $\lim_n\|p_n\|_1=0$ (when $(p_n)_n$ is thought of as a sequence in $L$).\\

Now, recall the Phillips' example, i.e. 
the   function \mbox{$f:[0,1]\to l^{\infty}([0,1])$}  defined as follows:
 \[ f(t)(s)=\left\{\begin{array}{lr}
1, & {\rm if\ } |t-s|\ \ {\rm is\  a\ dyadic\ rational}, \\
0.& {\rm otherwise.} \end{array}\right.
\]
As earlier proved by Phillips in \cite[Example 10.2]{phillips}, $f$ is not Birkhoff integrable, but it is McShane integrable, as later shown by Rodr\'iguez in \cite{r2009b}.
Since $l^{\infty}([0,1])$ is an $M$-space, $f$ is also (oM)-integrable.

Now, since the range of $f$ is contained in $L$ (endowed with the same ordering as $l^{\infty}$), by the previous remarks it follows
that $f$ is oH-integrable in this space, and therefore Bochner integrable thanks to Theorem \ref{Lspazio}.
So, the same mapping $f$ is Bochner integrable with respect to the $\|\cdot \|_1$-norm but not even Birkhoff integrable with respect to $\|\cdot \|_{\infty}$.\\

The example in \cite[Example 2.8]{bms} shows that also for this type of integral generally a.e. equality is not sufficient for two functions to have the same integral. However, the proof of Proposition \ref{bdd} shows that this is actually
 the case for bounded functions. For this reason here 
only bounded integrable mappings $f:[0,1]\to X$ are considered (unless otherwise specified).\\

One interesting result in this topic is that, at least for bounded functions, (oB)-integrability can also be defined by means of {\em Perron}-type partitions, i.e. (oB)-integrability is equivalent to (oBH)-integrability. 
 This is not easy to prove: to give an idea,  a similar equivalence 
for the (oM)-integral as defined at the beginning of this section, i.e. without requiring measurability of the involved gauges, is still an open problem. \\
In order to achieve the purpose,  some more notions are necessary. The following definition is inspired at \cite[Definition 2(b)]{nara}.

\begin{definition}\label{oriemann}\rm
Let $f:[0,1]\to X$ be any mapping, and fix any measurable subset $E\in [0,1]$. The function $f$ will be said to be {\em (o)-Riemann measurable} in $E$ if there exist an $(o)$-sequence $(p_n)_n$ in $X$ and a corresponding sequence $(F_n)_n$ of closed subsets of $E$, such that
$\lambda(E\setminus F_n)\leq n^{-1}$
and
\begin{eqnarray}\label{normadentro}
\sum_{i=1}^k|f(t_i)-f(t'_i)|\lambda(E_i)\leq p_n\end{eqnarray}
hold true, for every $n\in \enne$ and every finite collection of pairs $\{(t_i,E_i)_{i=1}^k\}$, where the sets $E_i$ are pairwise non-overlapping subintervals,
 with  the tags $t_i$ belong to $E_i\cap F_n$ for all $i$ and  \mbox{$\max\{\lambda(E_i)\}\leq n^{-1}$}.
\end{definition}

Remark that, in this definition, the formula (\ref{normadentro}) is (formally) quite stronger than the corresponding one, introduced in
 \cite[Definition 2(b)]{nara}: the norm in the latter is {\em outside} the summation, while here the modulus is  {\em inside}.\\

It is easy to see that, if $f$ is (o)-Riemann measurable on a set $E$, then also $\alpha f$ is, for every real constant $\alpha$. Moreover, if $f$ and $g$ are (o)-Riemann measurable on $E$, the same holds also for $f+g$. Other useful facts are collected in the following Proposition, whose proof is similar to that of \cite[Theorem 1]{nara}.

\begin{proposition}\label{mis}
Let $f:[0,1]\to X$ be any mapping, and fix any measurable set $E$ in $[0,1]$. The following properties hold:
\begin{description}
\item[(\ref{mis}.a)] If $f$ is {\rm (o)}-Riemann measurable on $E$, then it is on every measurable subset $E_1\subset E$.
\item[(\ref{mis}.b)] $f$ is {\rm (o)}-Riemann measurable on $E$ if and only if $f1_E$ is {\rm (o)}-Riemann measurable on $[0,1]$.
\end{description}
\end{proposition}
Moreover
\begin{theorem}\label{primonara}
Let $f:[0,1]\to X$ be any {\rm (oBH)}-integrable mapping on $[0,1]$. Then $f$ is {\rm (o)}-Riemann measurable.
\end{theorem}
\begin{proof}  
Thanks to the Henstock Lemma, in particular \ref{henstoc2}.2), it follows that, also for this type of integral, there exist an $(o)$-sequence  
 $(p_n)_n$ and a corresponding sequence $(\gamma_n)_n$ of measurable gauges such that
$$\sum_{i=1}^k|f(t_i)\lambda(E_i)-f(t_i')\lambda(E_i)|\leq p_n,$$
as soon as all the pairs $(t_i,E_i)_{i=1}^k$ form a $\gamma_n$-fine Henstock-type partition. (The fact that here the sets $E_i$ are subintervals is irrelevant).
Now, fix $n$. Since $\gamma_n$ is measurable, there exist a positive $\delta_n>0$ and an open set $G_n\subset [0,1]$ with $\lambda(G_n)<n^{-1}$ such that $\{t\in [0,1]: \gamma_n(t)<\delta_n\}\subset G_n$. Setting $F_n=[0,1]\setminus G_n$,  then $F_n$ is closed and, by the previous inequality,
$$\sum_{i=1}^k|f(t_i)-f(t'_i)|\lambda(E_i)\leq p_n$$
holds true, as soon as the intervals $E_i$ are non-overlapping and satisfy $\max_i\lambda(E_i)\leq \delta_n$, and the points $t_i,t'_i$ are in $E_i\cap F_n$.
\end{proof}

Now, it will be proven that, as soon as $f$ is bounded and (o)-Riemann measurable, then it is (oB)-integrable. Of course, thanks to Theorem \ref{primonara}, this will imply that (oBH)-integrability for bounded functions is equivalent to (oB)-integrability.
Proceeding as in \cite{nara},  a  technical Lemma concerning the  {\em oscillation} $\omega$ of a bounded function $f$ will be stated. Its  proof is omitted since it is completely similar to that of \cite[Lemma 1]{nara}.
\begin{lemma}\label{gordon1}
Let $f:[0,1]\to X$ be any bounded mapping, and fix any measurable set $F\subset [0,1]$. Assume that $a\in X$ and $\delta\in \erre$ are two positive elements such that
$$\sum_{i=1}^N(\beta_i-\alpha_i)\omega(f;[\alpha_i,\beta_i]\cap F)\leq \frac{a}{5}$$
as soon as $([\alpha_i,\beta_i])_{i=1}^N$ are pairwise disjoint subintervals of $[0,1]$ with $\max_i (\beta_i-\alpha_i)\leq 10\delta$.
Then 
$$\sum_{k=1}^K(\beta_k-\alpha_k)\omega(f;[\alpha_k-\delta,\beta_k+\delta]\cap F)\leq a$$
as soon as $([\alpha_k,\beta_k])_{k=1}^K$ are pairwise disjoint subintervals of $[0,1]$ with $\max_k (\beta_k-\alpha_k)\leq 2\delta$.
\end{lemma}
So it follows:
\begin{theorem}\label{naradue}
Let $f:[0,1]\to X$ be a bounded and {\rm (o)}-Riemann measurable map. Then $f$ is {\rm (oB)}-integrable.
\end{theorem}
\begin{proof}
First, denote by $M$ the supremum of $|f(x)|$, as $x$ runs in $[0,1]$. Next, let $(p_n)_n$ and $(F_n)_n$ denote the $(o)$-sequence and the corresponding sequence of measurable sets in $[0,1]$ related to the (o)-Riemann measurability of $f$. Now, define a new sequence $(\gamma'_n)_n$ of measurable gauges as follows:
$$\gamma'_n(t)=\left\{\begin{array}{ll}
\gamma_n(t)/20,&  t\in F_n\\
dist(t,F),&  t\notin F_n\end{array}\right.$$
Next, fix $n$ and choose two free partitions of $[a,b]$, say $P$ and $P'$, respectively with pairs $(E_i,t_i), i=1...k,$ and $(E'_j,t'_j), j=1,...k'$, subordinated to $\gamma'_n$. It holds
\begin{eqnarray*}
|\sigma(f,P)-\sigma(f,P')|&=&|\sum_{i,j}(f(t_i)-f(t'_j))\lambda(E_i\cap E'_j)|\leq \\ 
&\leq&
 |\sum_{(i,j)\in S}(f(t_i)-f(t'_j))\lambda(E_i\cap E'_j)|+
\\ &+&\sum_{(i,j)\notin S}|f(t_i)-f(t'_j)|\lambda(E_i\cap E'_j),
\end{eqnarray*}
where $S$ denotes the set of all couples $(i,j)$ such that $t_i$ and $t'_j$ belong to $F_n$.
By definition of the sets $F_n$, one has $\lambda(F_n^c)\leq n^{-1}$, and so, also thanks to the definition of $\gamma'_n$:
$$\sum_{(i,j)\notin S}|f(t_i)-f(t'_j)|\lambda(E_i\cap E'_j)\leq 2M n^{-1};$$
moreover, as to the first summand, one has clearly
$$|\sum_{(i,j)\in S}(f(t_i)-f(t'_j))\lambda(E_i\cap E'_j)|\leq \sum_{(i,j)\in S}|f(t_i)-f(t'_j)|\lambda(E_i\cap E'_j).$$
Now, from the definition of $\gamma'_n$, it follows that 
$$\max_{(i,j)\in S}\lambda(E_i\cap E'_j)\leq  5^{-1}\gamma_n :$$
 so, thanks to the Lemma \ref{gordon1} and (o)-Riemann measurability,
$$\sum_{(i,j)\in S}|f(t_i)-f(t'_j)|\lambda(E_i\cap E'_j)\leq 5p_n.$$
Now, from the previous inequalities, one gets
$$|\sigma(f,P)-\sigma(f,P')|\leq 5p_n+2M n^{-1}.$$
Since $(5p_n+2M n^{-1})_n$ is clearly an $(o)$-sequence, the Cauchy condition for the $(oB)$-integrability is satisfied; 
and, since $X$ is complete, $f$ turns out to be $(oB)$-integrable.
\end{proof}
\vskip.5cm
\noindent {\bf \large Conclusions}\\
In this paper the notions of Henstock (Mc Shane) integrability for
 functions defined in a metric compact regular space
 and taking values in a Banach lattice with an order-continuous norm are
 investigated. Both the norm-type and the order-type integrals have been
 examined. In general the
 order-type integral is stronger than the norm-one, while in $M$-spaces
 the two notions coincide and in $L$-spaces the order-type Henstock
 integral implies the Bochner one.\\

\noindent {\bf \large Acknowledgement} \\
The authors have been supported by University of Perugia -- Department of Mathematics and Computer Sciences-- Grant Nr 2010.011.0403,\\ 
Prin "Metodi logici per il trattamento dell'informazione",   Prin "Descartes" and by the Grant prot. U2014/000237 of GNAMPA - INDAM (Italy).

\end{document}